\documentclass[a4paper,12pt]{article}
\usepackage{amsmath, amsfonts, amssymb, amsthm}
\usepackage[english]{babel}

\newcounter{num}[section]

\newenvironment{theorem}
{\refstepcounter{num}%
\bigskip\noindent\nopagebreak[4]{\bf Theorem~\arabic{section}.\arabic{num}. }\it}

\newenvironment{corollary}
{\refstepcounter{num}%
\bigskip\noindent\nopagebreak[4]{\bf Corollary~\arabic{section}.\arabic{num}. }\it}

\newenvironment{lemma}
{\refstepcounter{num}%
\bigskip\noindent\nopagebreak[4]{\bf Lemma~\arabic{section}.\arabic{num}. }\it}

\newenvironment{remark}
{\refstepcounter{num}%
\bigskip\noindent\nopagebreak[4]{\bf Remark~\arabic{section}.\arabic{num}. }}

\newcommand{\N}{{\mathbb{N}}}

\newcommand{\LL}{{\mathcal{L}}}

\newcommand{\Ss}{{\mathbf{S}}}
\newcommand{\V}{{\mathrm{V}}}

\newcommand{\pr}{{\prime}}

\newcommand{\al}{{\alpha}}

\newcommand{\A}{{\mathcal{A}}}
\newcommand{\Bcal}{{\mathcal{B}}}

\renewcommand{\P}{{\mathbf{P}}}
\renewcommand{\a}{{\mathbf{a}}}

\renewcommand{\c}{{\mathbf{c}}}
\renewcommand{\d}{{\mathbf{d}}}

\renewcommand{\Pr}{{\mathrm{Pr}}}

\begin{document}

\title{Equations over direct powers of algebraic structures in relational languages}
\author{Artem N. Shevlyakov}
\maketitle
\abstract{We study equations over relational structures that approximate groups and semigroups. For such structures we proved the criteria, when a direct power of such algebraic structures is equationally Noetherian.}

\section{Introduction}

Let $\A$ be an algebraic structure with the universe $A$ of a functional language $\LL$. In other words, there are defined functions and constants over $\A$ that correspond to the symbols of $\LL$. One can define the structure $\Pr(\A)$ with the universe $A$ of a pure relational language $\LL_{pred}$ as follows: 
\begin{eqnarray}
R_f(x_1,\ldots,x_n,y)=\{(x_1,\ldots,x_n,y)\mid f(x_1,\ldots,x_n)=y\}\mbox{ in $\Pr(\A)$},\\
R_c(x)=\{x\mid x=c\}\mbox{ in $\Pr(\A)$},
\end{eqnarray} 
where functional and constant symbols $f,c$ belong to the language $\LL$.
Namely, the relation $R_f\in\LL_{pred}$ ($R_c\in\LL_{pred}$) is the graph of the function $f\in \LL$ (constant $c\in\LL$). 

The $\LL_{pred}$-structure $\Pr(\A)$ is called the {\it predicatization} of an $\LL$-structure $\A$. In particular, if $\A$ is a group of the language $\LL_g=\{\cdot,^{-1},1\}$ then $\Pr(\A)$ is an algebraic structure of the language $\LL_{g-pred}$ with~(\ref{eq:realtions_for_groups1},\ref{eq:realtions_for_groups3}). Notice that any equation over a group $\A$ may be rewritten in the language $\LL_{g-pred}$ by the introducing new variables. For example, the equation $x^{-1}y^{-1}xy=1$ has the following correspondence in the relational language $\LL_{g-pred}$:
\[
\Pr(\Ss)=\begin{cases}
I(x,x_1),\\
I(y,y_1),\\
M(x_1,y_1,z_1),\\
M(z_1,x,z_2),\\
M(z_2,y,z_3),\\
E(z_3)
\end{cases}
\]   
where the relations $I,M,E$ are defined by~(\ref{eq:realtions_for_groups1},\ref{eq:realtions_for_groups3}).

It is easy to see that the projection of the solution set of $\Ss$ onto the variables $x,y$ gives the solution set of the initial equation $x^{-1}y^{-1}xy=1$. More generally, {\it for any finite set of group equations $\Ss$ in variables $X$ there exists a system $\Pr(\Ss)$ of equations in the language $\LL_{g-pred}$ such that the solution set of $\Ss$ is the projection of the solution set $\Pr(\Ss)$ onto the variables $X$.} Hence, there arises the following important problem.

\medskip

\noindent {\bf Problem}. What properties of a finite system $\Ss$ ($\Pr(\Ss)$) are determined by the system $\Pr(\Ss)$ (respectively, $\Ss$)?

\medskip

This problem was originally studied in~\cite{shevl_herald_predicatiz}, where it was proved the general results for relational systems $\Pr(\Ss)$. 

The next principal problem is to describe relational structures $\Pr(\A)$ with ``hard'' and ``simple'' equational properties. According to~\cite{DMR1}, an algebraic structure $\Pr(\A)$ has ``simple'' equational properties if $\Pr(\A)$ is equationally Noetherian (i.e. any system of equations is equivalent over $\Pr(\A)$ to a finite subsystem). However, it was proved in~\cite{shevl_herald_predicatiz} that any algebraic structure of a {\it finite} relational language is equationally Noetherian. Thus, the Noetherian property gives a trivial classification of ``hard'' and ``simple'' relational structures $\Pr(\A)$. 

Therefore, we have to propose an alternative approach in the division of relational algebraic structures into classes with ``simple'' and ``hard'' equational properties. Our approach satisfies the following:
\begin{enumerate}
\item we deal with lattices of algebraic sets over a given algebraic structures (a set $Y$ is algebraic over a predicatization $\Pr(\A)$ if $Y$ is a solution set of an appropriate system of equations);
\item we use the common operations of model theory (direct products, substructures, ultra-products etc.); 
\item the partition into ``simple'' and ``hard'' algebraic structures is implemented by a list of first-order formulas $\Phi$ such that 
\begin{equation}
\A\mbox{ is ``simple'' }\Leftrightarrow \A\mbox{ satisfies }\Phi.
\label{eq:simple_is_Phi}
\end{equation}
In other words, the ``simple'' class of algebraic structures is axiomatizable by formulas $\Phi$. 
\end{enumerate}

Namely, we offer to consider infinite direct powers $\Pi \Pr(\A)$ of a predicatization $\Pr(\A)$ and study Diophantine equations over $\Pi\Pr(\A)$ instead of Diophantine equations over $\Pr(\A)$ (an equation $E(X)$ is said to be Diophantine over an algebraic structure $\Bcal$ if $E(X)$ may contain the occurrences of any element of $\Bcal$). The decision rule in our approach is the following:
\begin{equation}
\Pr(\A)\mbox{ is ``simple'' }\Leftrightarrow \mbox{all direct powers of }\Pr(\A)\mbox{ are equationally Noetherian };
\label{eq:our_decision_rule}
\end{equation}
otherwise, an algebraic structure $\Pr(\A)$ is said to be ``hard''.   

Some results of the type~(\ref{eq:simple_is_Phi}) and~(\ref{eq:our_decision_rule}) were obtained in~\cite{shevl_shah}, where we found formulas $\Phi$ for the classes of groups, rings and monoids in functional languages. For example, a group (ring) has a ``simple'' equational theory in the functional language iff it is abelian (respectively, with zero multiplication). 

On the other hand, we prove below that any group in the language $\LL_{g-pred}$ has equationally Noetherian direct powers (Corollary~\ref{cor:for_groups}). Moreover, the similar result holds for the natural generalizations of groups: quasi-groups and loops (Remark~\ref{rem:loops}). 

However, the class of semigroups has a nontrivial classification~(\ref{eq:our_decision_rule}). We find two quasi-identities~(\ref{eq:semigroup_QI1},\ref{eq:semigroup_QI2}) such that a semigroup $S$ satisfies~(\ref{eq:semigroup_QI1},\ref{eq:semigroup_QI2}) iff any direct power of $\Pr(S)$ is equationally Noetherian (Theorem~\ref{th:pred_semigroup_N}).

In the class of finite semigroups the conditions~(\ref{eq:semigroup_QI1},\ref{eq:semigroup_QI2}) imply that the minimal ideal (kernel) of a semigroup $S$ is a rectangular band of groups, and the kernel coincides with the ideal of reducible elements of $S$ (Theorem~\ref{th:semigroup_description}). However, if the kernel of a finite semigroup $S$ is a group then the conditions of Theorem~\ref{th:pred_semigroup_N} become sufficient for Noetherian property of any direct power $\Pi\Pr(S)$.  


\section{Basic notions}

In the current paper we deal with relational languages that interpret functions and constants in groups and semigroups.

Let $S$ be a semigroup. One can define the language $\LL_{s-pred}=\{M^{(3)}\}$ and a relation 
\begin{equation*}
M(x,y,z)\Leftrightarrow xy=z.
\label{eq:relation_for_semigroups}
\end{equation*}

Any group $G$ may be considered as an algebraic structure of the relational language $\LL_{g-pred}=\{M^{(3)},I^{(2)},E^{(1)}\}$, where
\begin{eqnarray}
\label{eq:realtions_for_groups1}
M(x,y,z)\Leftrightarrow xy=z,\\
I(x,y)\Leftrightarrow x=y^{-1},\\
E(x)\Leftrightarrow x=1.
\label{eq:realtions_for_groups3}
\end{eqnarray}

An algebraic structure of the language $\LL_{s-pred}$ ($\LL_{g-pred}$) is called the {\it predicatization} of a semigroup $S$ (group $G$) if the operations over $S$ ($G$) corresponds to the relations~(\ref{eq:realtions_for_groups1},\ref{eq:realtions_for_groups3}). The predicatization of a semigroup $S$ (group $G$) is denoted by $\Pr(S)$ (respectively, $\Pr(G)$).

Following~\cite{DMR2}, we give the main definitions of algebraic geometry over algebraic structures (below $\LL\in\{\LL_{s-pred},\LL_{p-pred}\}$). 

An {\it equation over  $\LL$ ($\LL$-equation)} is an atomic formula over $\LL$.
The examples of equations are the following: $M(x,x,x)$, $M(x,y,x)$ ($\LL_{s-pred}$-equations); $M(x,x,y)$, $I(x,y)$, $I(x,x)$, $E(x)$ ($\LL_{g-pred}$-equations).

A {\it system of $\LL$-equations} ({\it $\LL$-system} for shortness) is an arbitrary set of $\LL$-equations. { Notice that we will consider only systems in a finite set of variables $X=\{x_1,x_2,\ldots,x_n\}$}. The set of all solutions of $\Ss$ in an $\LL$-structure $\A$ is denoted by $\V_\A(\Ss)\subseteq \A^n$. A set $Y\subseteq \A^n$ is said to be an {\it algebraic set over $\A$} if there exists an $\LL$-system $\Ss$ with $Y=\V_\A(\Ss)$. If the solution set of an $\LL$-system $\Ss$ is empty, $\Ss$ is said to be {\it inconsistent}. Two $\LL$-systems $\Ss_1,\Ss_2$ are called {\it equivalent over an $\LL$-structure $\A$} if $\V_\A(\Ss_1)=\V_\A(\Ss_2)$.

An $\LL$-structure $\A$ is {\it $\LL$-equationally Noetherian} if any infinite $\LL$-system $\Ss$ is equivalent over $\A$ to a finite subsystem $\Ss^\pr\subseteq \Ss$. 

Let $\A$ be an $\LL$-structure. By $\LL(\A)$ we denote the language $\LL\cup\{a\mid a\in \A\}$ extended by new constants symbols which correspond to elements of $\A$. The language extension allows us to use constants in equations. The examples of equations in the extended languages are the following: $M(x,y,a)$ ($\LL_{s-pred}(S)$-equation and $a\in S$); $M(a,x,b)$, $I(x,a)$, $E(a)$ ($\LL_{g-pred}(S)$-equations and $a,b\in G$). Obviously, the class of $\LL(\A)$-equations is wider than the class of $\LL$-equations, so an $\LL$-equationally Noetherian algebraic structure $\A$ may lose this property in the language $\LL(\A)$.

Since the algebraic structures $\A$ and $\Pr(\A)$ have the same universe, we will write below $\V_\A(\Ss)$  ($\LL(\A)$) instead of $\V_{\Pr(\A)}(\Ss)$ (respectively, $\LL(\Pr(\A))$).

Let $\A$ be a relational $\LL$-structure. The {\it direct power} $\Pi \A=\prod_{i\in I}\A$ of $\A$ is the set of all sequences $[a_i\mid i\in I]$ and any relation $R\in\LL$ is defined as follows
\[
R([a^{(1)}_i\mid i\in I],[a^{(2)}_i\mid i\in I],\ldots,[a^{(n)}_i\mid i\in I])\Leftrightarrow 
R(a^{(1)}_i,a^{(2)}_i,\ldots,a^{(n)}_i) \mbox{ for each }i\in I.
\]
A map $\pi_k\colon\Pi\A\to \A$ is called the {\it projection onto the $i$-th coordinate} if $\pi_k([a_i\mid i\in I])=a_k$.

Let $E(X)$ be an $\LL(\Pi\A)$-equation over a direct power $\Pi\A$.  We may rewrite $E(X)$ in the form $E(X,\overrightarrow{\mathbf{C}})$, where $\overrightarrow{\mathbf{C}}$ is an array of constants occurring in the equation $E(X)$. One can introduce the {\it projection of an equation} onto the $i$-th coordinate as follows:
\[
\pi_i(E(X))=\pi_i(E(X,\overrightarrow{\mathbf{C}}))=E(X,\pi_i(\overrightarrow{\mathbf{C}})),
\] 
where $\pi_i(\overrightarrow{\mathbf{C}})$ is an array of the $i$-th coordinates of the elements from $\overrightarrow{\mathbf{C}}$. For example, the $\LL_{s-pred}(\Pi\A)$-equation $M(x,[a_1,a_2,a_3,\ldots],[b_1,b_2,b_3,\ldots])$ has the following projections
\begin{eqnarray*}
M(x,a_1,b_1),\\
M(x,a_2,b_2),\\
M(x,a_3,b_3),\\
\ldots
\end{eqnarray*}
Obviously, any projection of an $\LL(\Pi\A)$-equation is an $\LL(\A)$-equation.

Let us take an $\LL(\Pi\A)$-system $\Ss=\{E_j(X)\mid j\in J\}$. The $i$-th projection of $\Ss$ is the $\LL(\A)$-system defined by $\pi_i(\Ss)=\{\pi_i(E_j(X))\mid j\in J\}$. The projections of an $\LL(\Pi\A)$-system $\Ss$ allow to describe the solution set of $\Ss$ by
\begin{equation}
\V_{\Pi\A}(\Ss)=\{[P_i\mid i\in I]\mid P_i\in \V_\A(\pi_i(\Ss))\}.
\label{eq:solution_set_via_projections}
\end{equation}
In particular, if one of the projections $\pi_i(\Ss)$ is inconsistent, so is $\Ss$.

The following statement immediately follows from the description~(\ref{eq:solution_set_via_projections}) of the solution set over a direct powers.

\begin{lemma}
Let $\Ss=\{E_j(X)\mid j\in J\}$ be an $\LL(\Pi\A)$-system over $\Pi\A$. If one of the projections $\pi_i(\Ss)$ is inconsistent, so is $\Ss$. Moreover, if $\A$ is $\LL$-equationally Noetherian, then an inconsistent $\LL(\Pi\A)$-system $\Ss$ is equivalent to a finite subsystem.
\label{l:inconsistent_systems}
\end{lemma}
\begin{proof}
The first assertion directly follows from~(\ref{eq:solution_set_via_projections}). Suppose $\A$ is $\LL$-equationally Noetherian, and $\pi_i(\Ss)$ is inconsistent. Hence, $\pi_i(\Ss)$ is equivalent to its finite inconsistent subsystem $\{\pi_i(E_j(X))\mid j\in J^\pr\}$, $|J^\pr|<\infty$, and the finite subsystem $\Ss^\pr=\{E_j(X)\mid j\in J^\pr\}\subseteq\Ss$ is also inconsistent.
\end{proof}

\section{Predicatization of semigroups and groups}

\begin{theorem}
Let $\Pr(S)$ be the predicatization of a semigroup $S$. A direct power of $\Pr(S)$ is equationally Noetherian iff the following quasi-identities 
\label{th:pred_semigroup_N}
\begin{eqnarray}
\label{eq:semigroup_QI1}
\forall a\forall b\forall\alpha\forall\beta \left((a\alpha=a\beta)\to(b\al=b\beta)\right),\\
\label{eq:semigroup_QI2}
\forall a\forall b\forall\alpha\forall\beta \left((\alpha a=\beta a)\to(\al b=\beta b)\right)
\end{eqnarray} 
hold in $S$.
\end{theorem}
\begin{proof}
First, we prove the ``if'' part of the theorem. Suppose $S$ satisfies~(\ref{eq:semigroup_QI1},\ref{eq:semigroup_QI2}) and consider an infinite  $\LL_{s-pred}(\Pi S)$-system $\Ss$. One can represent $\Ss$ as a finite union of the following systems
\begin{equation}
\label{eq:decomposition_semigroups}
\Ss=\bigcup_{1\leq i,j\leq n}\Ss_{cij}\bigcup_{1\leq i,j\leq n}\Ss_{icj}
\bigcup_{1\leq i,j\leq n}\Ss_{ijc}
\bigcup_{1\leq i\leq n}\Ss_{cci}\bigcup_{1\leq i\leq n}\Ss_{cic}\bigcup_{1\leq i\leq n}\Ss_{icc}\bigcup\Ss_0,
\end{equation}
where each equation of $\Ss_0$ is one of the following types:
\begin{enumerate}
\item $x_i=x_j$; 
\item $x_i=\c_j$;
\item $\c_i=\c_j$;
\item $M(x_i,x_j,x_k)$;
\end{enumerate}
and $\Ss_{cij}=\{M(\c_k,x_i,x_j)\mid k\in K\}$, $\Ss_{icj}=\{M(x_i,\c_k,x_j)\mid k\in K\}$, $\Ss_{ijc}=\{M(x_i,x_j,\c_k)\mid k\in K\}$, $\Ss_{cci}=\{M(\c_k,\d_k,x_i)\mid k\in K\}$, $\Ss_{cic}=\{M(\c_k,x_i,\d_k)\mid k\in K\}$, $\Ss_{icc}=\{M(x_i,\c_k,\d_k)\mid k\in K\}$ ($\c_k,\d_k\in\Pi\Pr(S)$), where each system above has its own index set $K$.

Clearly, the system $\Ss_0$ is equivalent to its finite subsystem. So it is sufficient to prove that each of other systems is equivalent to a finite subsystem over $\Pi S$. According to Lemma~\ref{l:inconsistent_systems}, we may assume that any system below is consistent.

Thus, we have the following cases.
\begin{enumerate}
\item Let $\Ss_{icc}=\{M(x_i,\c_k,\d_k)\mid i\in I\}$ and $M(x_i,\c_1,\d_1)$ be an arbitrary equation of $\Ss_{icc}$. Since $\Ss_{icc}$ is consistent then one can choose $\bar{\al}\in\V_{\Pi S}(\Ss_{icc})$, $\bar{\beta}\in\V_{\Pi Ss}(M(x_i,\c_1,\d_1))$. We have $\bar{\al}\c_1=\bar{\beta}\c_1=\d_1$. By the quasi-identities~(\ref{eq:semigroup_QI1},\ref{eq:semigroup_QI2}), $\bar{\al}\c_k=\bar{\beta}\c_k$ for any $\c_k$. Hence, $\bar{\beta}$ satisfies all equations from $\Ss_{icc}$, and $\Ss_{icc}$ is equivalent to the equation $M(x_i,\c_1,\d_1)$. The proof for the systems $\Ss_{cic},\Ss_{cci}$ is similar.


\item Let $\Ss_{icj}=\{M(x_i,\c_k,x_j)\mid i\in I\}$ (the proof for $\Ss_{cij},\Ss_{ijc}$ is similar). Since $\Ss_{icj}$ is consistent, there exist a point $(\bar{\al},\bar{\beta})\in\V_{\Pi S}(\Ss_{icj})$ and the equalities $\bar{\al}\c_k=\bar{\al}\c_l=\bar{\beta}$ hold for any $k,l\in K$. By~(\ref{eq:semigroup_QI1},\ref{eq:semigroup_QI2}), for any $\bar{\gamma}\in\Pi S$ it holds $\bar{\gamma}\c_k=\bar{\gamma}\c_l$. Thus, the solution set of $\Ss_{icj}$ is $Y=\{(\bar{\gamma},\bar{\gamma}\c_1)\mid \bar{\gamma}\in\Pi S\}$ and $\Ss_{icj}$ is equivalent to the equation $x\c_1=y$. 

\end{enumerate}   

\bigskip

Now we prove the ``only if'' part of the theorem. Suppose the quasi-identity~(\ref{eq:semigroup_QI1}) does not hold in $S$ (for the formula~(\ref{eq:semigroup_QI2}) the proof is similar). It follows there exist elements $a,b,\al,\beta$ such that $a\al=a\beta=c$, $b\al\neq b\beta$. Let us consider the  system 
\[
\Ss=\{M(\a_n,x,\c_n)\mid n\in \N\},
\]
where 
\[
\a_n=[\underbrace{a,\ldots,a}_{\mbox{$n$ times}},b,b,\ldots],
\c_n=[\underbrace{c,\ldots,c}_{\mbox{$n$ times}},b\al,b\al,\ldots].
\]
One can directly check that the point
\[
\a=[\underbrace{\beta,\ldots,\beta}_{\mbox{$n$ times}},\al,\al,\ldots]
\]
satisfies the first $n$ equations of $\Ss$. However the $(n+1)$-th equation of $\Ss$ gives $\a_{n+1}\a\neq\c_{n+1}$, since its $(n+1)$-th projection defines the equation $bx=b\al$, but $b\al\neq b\beta$. Thus, $\Ss$ is not equivalent to any finite subsystem.
\end{proof}

\begin{corollary}
Let $\Pr(G)$ be the predicatization of a group $G$. Then any direct power of $\Pr(G)$ is $\LL_{g-pred}(\Pi G)$equationally Noetherian.
\label{cor:for_groups}
\end{corollary}
\begin{proof}
Since the equality $a\al=a\beta$ ($\al a=\beta a$) implies $\al=\beta$ in any group, the quasi-identities~(\ref{eq:semigroup_QI1},\ref{eq:semigroup_QI2}) obviously hold in $G$.
Thus, any infinite system of the form $\{M(\ast,\ast,\ast)\mid i\in I\}$ is equivalent to a finite subsystem.

One can directly prove that for any group infinite systems of the form $\{I(\ast,\ast)\mid i\in I\}$ and $\{E(\ast)\mid i\in I\}$ are also equivalent to their finite subsystems over $\Pi G$.

Thus, any system of $\LL_{g-pred}(\Pi G)$-equations is equivalent over $\Pi G$ to its finite subsystem.
\end{proof}

\begin{remark}
\label{rem:loops}
The last corollary also holds for quasi-groups. Notice that a quasi-group is a non-associative analogue of a group. Any quasi-group admits the analogues of the group divisibility, hence  the quasi-identities~(\ref{eq:semigroup_QI1},\ref{eq:semigroup_QI2}) obviously hold in any quasi-group. Thus, any direct power of a quasi-group $G$ is $\LL_{s-pred}(\Pi G)$-equationally Noetherian (notice here we consider quasi-groups and loops in the language $\LL_{s-pred}$, since not any quasi-group admits the relations $I(x,y)$ and $E(x)$). 
\end{remark}

\bigskip

Below we study finite semigroups $S$ that satisfy Theorem~\ref{th:pred_semigroup_N}. 

A subset $I\subseteq S$ is called a \textit{left (right) ideal} if for any $s\in S$, $a\in I$ it holds $sa\in I$ ($as\in I$). An ideal which is right and left simultaneously is said to be {\it two-sided} (or an {\it ideal} for shortness). 

A semigroup $S$ with a unique ideal $I=S$ is called \textit{simple}. Let us remind the classical Sushkevich-Rees theorem for finite simple semigroups.  

\begin{theorem}
\label{th:sushkevic_rees}
For any finite simple semigroup $S$ there exists a finite group $G$ and finite sets $I,\Lambda$ such that  $S$ is isomorphic to the set of triples $(\lambda,g,i)$, $g\in G$, $\lambda\in\Lambda$, $i\in I$. The multiplication over the triples $(\lambda,g,i)$ is defined by
\[
(\lambda,g,i)(\mu,h,j)=(\lambda,gp_{i\mu}h,j),
\]
where $p_{i\mu}\in G$ is an element of a matrix $\P$ such that
\begin{enumerate}
\item $\P$ consists of  $|I|$ rows and $|\Lambda|$ columns;
\item the elements of the first row and the first column equal  $1\in G$ (i.e. $\P$ is {\it normalized}).
\end{enumerate}
\end{theorem}

Following Theorem~\ref{th:sushkevic_rees}, we denote any finite simple semigroup $S$ by $S=(G,\P,\Lambda,I)$. 

%
%

The minimal ideal of a semigroup $S$ is called a \textit{kernel} and denoted by $Ker(S)$ (any finite semigroup always has a unique kernel). Obviously, if $S=Ker(S)$ the semigroup is simple. If $Ker(S)$ is a group then $S$ is said to be a \textit{homogroup}. The next theorem contains the necessary information about homogroups.

\begin{theorem}\textup{\cite{lyapin}}
\label{th:homogroups_properties}
In a homogroup $S$ the identity element $e$ of the kernel $Ker(S)$ is idempotent ($e^2=e$) and belongs to the center of $S$ (i.e. $e$ commutes with any $s\in S$).
\end{theorem}

\medskip

A semigroup $S$ is called a {\it rectangular band of groups} if $S=(G,\P,\Lambda,I)$ and $p_{i\lambda}=1$ for any $i\in I$, $\lambda\in \Lambda$.

\begin{lemma}
Suppose a finite simple semigroup $S$ satisfies~(\ref{eq:semigroup_QI1},\ref{eq:semigroup_QI2}). Then $S$ is a rectangular band of groups. 
\label{l:Noeth_implies_homogroup_or_band}
\end{lemma}
\begin{proof}
By Theorem~\ref{th:sushkevic_rees}, $S=(G,\P,\Lambda,I)$ for some finite group $G$, matrix $\P$ and finite sets of indexes $\Lambda,I$. 

Assume that $|\Lambda|>1$ and $p_{i\lambda}\neq 1$ for some $i,\lambda$.

Let $a=(1,1,1)$, $\al=(\lambda,1,1)$, $\beta=(1,1,1)$ and hence 
\begin{equation}
a\al=(1,1,1)(\lambda,1,1)=(1,1,1)=(1,1,1)(1,1,1)=a\beta.
\label{eq:1111111}
\end{equation}
However, for $b=(1,1,i)$ we have
\begin{eqnarray}
b\al=(1,1,i)(\lambda,1,1)=(1,p_{i\lambda},1)\neq (1,1,1)=(1,1,i)(1,1,1)=b\beta.
\label{eq:22222222}
\end{eqnarray}
We obtain that the equalities~(\ref{eq:1111111},\ref{eq:22222222}) contradict~(\ref{eq:semigroup_QI1},\ref{eq:semigroup_QI2}).

Thus, either $p_{i\lambda}=1$ for all $i,\lambda$ or $|\Lambda|=|I|=1$. In any case $S$ is a rectangular band of groups.
\end{proof}

An element $s$ of a semigroup $S$ is called {\it reducible} if there exist $a,b\in S$ with $s=ab$. Clearly, the set of all reducible elements $Red(S)$ is an ideal of a semigroup $S$.

\begin{lemma}
Let $S$ be a finite semigroup satisfying~(\ref{eq:semigroup_QI1},\ref{eq:semigroup_QI2}). Then $Ker(S)$ is the set of all reducible elements. 
\label{l:reducible_elements_rect_band}
\end{lemma}
\begin{proof}
Let $b\in S$. We have $(\lambda,g,i)b=(\lambda,g,i)(1,1,i)b=(\lambda,g,i)r$, where $r=(1,1,i)b\in Ker(S)$. By~(\ref{eq:semigroup_QI1}), we obtain $ab=ar$ for any $a\in S$. Since $ar\in Ker(S)$, so is $ab$. Thus, any product of elements belongs to $Ker(S)$. Thus, $Red(S)=Ker(S)$.
\end{proof}

\begin{theorem}
If a direct power $\Pi Pr(S)$ of a finite semigroup $S$ is equationally Noetherian, then $Ker(S)=Red(S)$ and $Ker(S)$ is a rectangular band of groups. 
\label{th:semigroup_description}
\end{theorem}
\begin{proof}
The proof immediately follows from Lemmas~\ref{l:Noeth_implies_homogroup_or_band},~\ref{l:reducible_elements_rect_band}.
\end{proof}
%
%
%

However, homogroups satisfy the converse statement of Theorem~\ref{th:semigroup_description}.  

\begin{theorem}
If $Ker(S)=Red(S)$ for a homogroup $S$, then the direct power $\Pi S$ is $\LL_{s-pred}(\Pi S)$-equationally Noetherian. 
\label{th:converse_theorem_for_homogroups}
\end{theorem}
\begin{proof}
Let us take $a,b,\al,\beta$ such that $a\al=a\beta$, and $e$ be the identity of $Ker(S)$. We have
\begin{eqnarray*}
a\al&=&a\beta\mid \cdot e\\
ea\al&=&ea\beta\\
(ea)\al&=&(ea)\beta\mid \cdot (ea)^{-1}\mbox{ since $ea$ belongs to the group $Ker(S)$}\\
e\al&=&e\beta\mid \mbox{$e$ is a central element}\\
\al e&=&\beta e.
\end{eqnarray*}

We have (below we use $b\beta\in Ker(S)=Red(S)$):
\begin{equation*}
b\al=(b\al)e=b(\al e)=b(\beta e)=(b\beta)e=b\beta
\end{equation*}
Thus, the quasi-identity~(\ref{eq:semigroup_QI1}) holds for $S$. The proof for the quasi-identity~(\ref{eq:semigroup_QI2}) is similar. 
\end{proof}

One can directly to check that for a rectangular band of groups  $S=(G,\P,\Lambda,I)$ the converse statement of Theorem~\ref{th:semigroup_description} also holds. 

Thus, one can formulate the following conjecture.

\medskip

\noindent {\bf Conjecture.} If a finite semigroup $S$ has a  rectangular band of groups $Ker(S)=Red(S)$ then $S$ satisfies the quasi-identities~(\ref{eq:semigroup_QI1},\ref{eq:semigroup_QI2}).

The information of the author:

Artem N. Shevlyakov

Sobolev Institute of Mathematics

644099 Russia, Omsk, Pevtsova st. 13

\medskip

Omsk State Technical University

pr. Mira, 11, 644050

Phone: +7-3812-23-25-51.

e-mail: \texttt{a\_shevl@mail.ru}

\end{document}